\documentclass[12pt]{amsart}
\usepackage{url,amsmath,amssymb,amsthm,amsthm,mathtools,enumitem}
\calclayout
\DeclareFontEncoding{OT2}{}{}
\DeclareFontSubstitution{OT2}{cmr}{m}{n}
\DeclareFontFamily{OT2}{cmr}{\hyphenchar\font45 }
\DeclareFontShape{OT2}{cmr}{m}{n}{
<5><6><7><8><9>gen*wncyr
<10><10.95><12><14.4><17.28><20.74><24.88>wncyr10}{}
\DeclareFontShape{OT2}{cmr}{b}{n}{
<5><6><7><8><9>gen*wncyb
<10><10.95><12><14.4><17.28><20.74><24.88>wncyb10}{}
\DeclareMathAlphabet{\mathcyr}{OT2}{cmr}{m}{n}
\DeclareMathAlphabet{\mathcyb}{OT2}{cmr}{b}{n}
\SetMathAlphabet{\mathcyr}{bold}{OT2}{cmr}{b}{n}
\newtheorem{theoremcounter}{Theorem Counter}[section]
\theoremstyle{plain}
\newtheorem{theorem}[theoremcounter]{Theorem}
\newtheorem{proposition}[theoremcounter]{Proposition}
\newtheorem{lemma}[theoremcounter]{Lemma}
\theoremstyle{remark}

\newcommand{\fH}{\mathfrak{H}}
\newcommand{\bk}{\boldsymbol{k}}
\newcommand{\bl}{\boldsymbol{l}}

\newcommand{\Li}{\mathrm{Li}}

\newcommand{\QQ}{\mathbb{Q}}
\newcommand{\RR}{\mathbb{R}}
\newcommand{\ZZ}{\mathbb{Z}}
\newcommand{\wt}{\mathrm{wt}}
\newcommand{\sh}{\mathbin{\mathcyr{sh}}}
\numberwithin{equation}{section}
\address{Department of Mathematical Sciences, Aoyama Gakuin University, 5-10-1 Fuchinobe, Chuo-ku, Sagamihara, Kanagawa, 252-5258, Japan}
\email{seki@math.aoyama.ac.jp}

\thanks{This research was supported by JSPS KAKENHI Grant Number 21K13762.}

\title[A proof of the EDSR without using integrals]{A proof of the extended double shuffle relation without using integrals}
\author{Shin-ichiro Seki}
\date{}

\begin{document}

\maketitle

\begin{abstract}
We present a new proof of the extended double shuffle relation for multiple zeta values which notably does not rely on the use of integrals.
This proof is based on a formula recently obtained by Maesaka, Watanabe, and the author.
\end{abstract}

\section{Introduction}
The \emph{multiple zeta value} (= MZV) $\zeta(\bk)$ is defined as
\[
\zeta(\bk)\coloneqq\sum_{0<n_1<\cdots<n_r}\frac{1}{n_1^{k_1}\cdots n_r^{k_r}},
\]
where $\bk=(k_1,\dots, k_r)$ is a tuple of positive integers, and $k_r>1$ to ensure convergence.
Such a tuple $\bk$ is called an \emph{admissible index}.
The MZV can be expressed in terms of two kinds of limits, as described below.
Let $N$ be a positive integer and $z$ a real number satisfying $0<z<1$.
Let $\bk=(k_1,\dots, k_r)$ be a tuple of positive integers that allows for $k_r=1$, which is referred to as an \emph{index}.
We call $\wt(\bk)\coloneqq k_1+\cdots+k_r$ its \emph{weight}.
Then the \emph{multiple harmonic sum} $\zeta_{<N}^{}(\bk)$ is defined as
\[
\zeta_{<N}^{}(\bk)\coloneqq\sum_{0<n_1<\cdots<n_r<N}\frac{1}{n_1^{k_1}\cdots n_r^{k_r}}
\]
and the \emph{multiple polylogarithm} $\mathrm{Li}_{\bk}(z)$ is defined as
\[
\Li_{\bk}(z)\coloneqq\int_{0<t_1<\cdots<t_k<z}\omega_{u_1}(t_1)\cdots\omega_{u_k}(t_k),
\]
where $k\coloneqq\wt(\bk)$, $\omega_0(t)\coloneqq\frac{\mathrm{d}t}{t}$, $\omega_1(t)\coloneqq\frac{\mathrm{d}t}{1-t}$, and if $i\in J(\bk)$, then $u_i=1$; otherwise, $u_i=0$.
Here, $J(\bk)$ is defined as $J(\bk)\coloneqq\{1,k_1+1, k_1+k_2+1,\dots, k_1+\cdots+k_{r-1}+1\}$.
If $\bk$ is admissible, then it is known that the following two limit expressions can be obtained (the first is trivial by definition):
\[
\lim_{N\to\infty}\zeta_{<N}^{}(\bk)=\zeta(\bk)=\lim_{z\to1}\mathrm{Li}_{\bk}(z).
\]
This two-sided nature, expressed by both the series and the iterated integral, gives rise to the rich structures of the space generated by the MZVs.
To describe the \emph{double shuffle relation} (= DSR), which is a prototypical example of such structures, we prepare the language of the so-called \emph{Hoffman algebras}.

Let $\QQ\langle e_0, e_1\rangle$ be a non-commutative polynomial algebra in two variables $e_0$ and $e_1$.
We define $\fH^1$ as $\QQ+e_1\QQ\langle e_0, e_1\rangle$ and its subspace $\fH^0$ as $\QQ+e_1\QQ\langle e_0, e_1\rangle e_0$.
Let $e_k\coloneqq e_1e_0^{k-1}$ for each positive integer $k$.
For each index $\bk=(k_1,\dots, k_r)$, $e_{\bk}\coloneqq e_{k_1}\cdots e_{k_r}\in\fH^1$.
If $\bk$ is admissible, then $e_{\bk}\in\fH^0$.
We define a $\QQ$-linear map $Z\colon\fH^0\to\RR$ by $Z(1)=1$ and $Z(e_{\bk})=\zeta(\bk)$.
The \emph{harmonic product} $*$ on $\fH^1$ is defined by rules $w*1=1*w=w$ for any word $w\in\fH^1$, and 
\[
we_{k_1}*w'e_{k_2}=(w*w'e_{k_2})e_{k_1}+(we_{k_1}*w')e_{k_2}+(w*w')e_{k_1+k_2}
\]
for any words $w, w'\in\fH^1$, $k_1, k_2\in\ZZ_{>0}$, with $\QQ$-bilinearity.
The \emph{shuffle product} $\sh$ on $\QQ\langle e_0, e_1\rangle$ is defined by rules $w\sh 1=1\sh w=w$ for any word $w\in\QQ\langle e_0, e_1\rangle$, and
\[
wu_1\sh w'u_2=(w\sh w'u_2)u_1+(wu_1\sh w')u_2
\]
for any words $w, w'\in\QQ\langle e_0, e_1\rangle$, $u_1, u_2\in\{e_0, e_1\}$, with $\QQ$-bilinearity.
Then $\fH^1$ becomes a commutative $\QQ$-algebra with respect to $*$ (resp.~$\sh$), which is denoted by $\fH^1_*$ (resp.~$\fH^1_{\sh}$).
The subspace $\fH^0$ of $\fH^1$ is closed under $*$ (resp.~$\sh$) and becomes a $\QQ$-subalgebra of $\fH^1_*$ (resp.~$\fH^1_{\sh}$), which is denoted by $\fH^0_{*}$ (resp.~$\fH^0_{\sh}$).

A given product of two MZVs can be expressed in two distinct ways as a linear combination of MZVs, utilizing their series and integral expressions, respectively.
The products $*$ and $\sh$ have been defined earlier to ensure that the following holds (the harmonic product corresponds to the decomposition by series, and the shuffle product corresponds to the decomposition by integrals): for any $w, w'\in\fH^0$,
\[
Z(w*w')=Z(w)Z(w')=Z(w\sh w').
\]
The relation among MZVs
\[
Z(w*w'-w\sh w')=0
\]
is called the double shuffle relation (= DSR).
The DSR provides a broad family of relations among MZVs, yet there exist relations, such as Euler's $\zeta(3)=\zeta(1,2)$, that cannot be derived from the DSR.
Let us briefly review the content of the \emph{extended double shuffle relation} (= EDSR) obtained by Ihara, Kaneko, and Zagier to overcome this.

Let $T$ be an indeterminate and $\RR[T]$ the usual polynomial ring.
Then there exist two algebra homomorphisms $Z^*\colon\fH^1_*\to\RR[T]$ and $Z^{\sh}\colon\fH^1_{\sh}\to\RR[T]$ that are uniquely characterized by the properties that both extend $Z\colon\fH^0\to\RR$ and send $e_1$ to $T$ (\cite[Proposition~1]{IharaKanekoZagier2006}).
We also use $Z_{\bk}^*(T)\coloneqq Z^*(e_{\bk})$ and $Z_{\bk}^{\sh}(T)\coloneqq Z^{\sh}(e_{\bk})$ for an index $\bk$.
By using these regularized polynomials, it is possible to describe the asymptotic behavior of $\zeta_{<N}^{}(\bk)$ as $N \to\infty$ and the asymptotic behavior of $\Li_{\bk}(z)$ as $z\to 1$, even when the index $\bk$ is not admissible.
The notation $\log^{\bullet}N$ implies the existence of some positive integer $a$ that does not depend on $N$, represented as $\log^aN$.
Similarly for $\log^{\bullet}(1-z)$.
Furthermore, let $\gamma$ denote Euler's constant.
\begin{proposition}[Ihara--Kaneko--Zagier~\cite{IharaKanekoZagier2006}]\label{prop:asymp}
For any index $\bk$, we have
\begin{equation}\label{eq:asympH}
\zeta_{<N}^{}(\bk)=Z^*_{\bk}(\log N+\gamma)+O(N^{-1}\log^{\bullet}N),\quad \text{as } N\to\infty
\end{equation}
and
\begin{equation}\label{eq:asympLi}
\mathrm{Li}_{\bk}(z)=Z^{\sh}_{\bk}(-\log(1-z))+O((1-z)|\log^{\bullet}(1-z)|),\quad \text{as } z\to1.
\end{equation}
\end{proposition}
By comparing these two kinds of asymptotic behaviors through analysis using the gamma function, they have proved the following fundamental theorem.
\begin{theorem}[{Regularization theorem; Ihara--Kaneko--Zagier~\cite[Theorem~1]{IharaKanekoZagier2006}}]\label{thm:regularization}
For any index $\bk$, we have
\begin{equation}\label{eq:regularization}
Z_{\bk}^{\sh}(T)=\rho(Z_{\bk}^*(T)),
\end{equation}
where $\rho\colon\RR[T]\to\RR[T]$ is the $\RR$-linear map defined in \cite[(2.2)]{IharaKanekoZagier2006}.
\end{theorem}
By combining this theorem with the DSR, the DSR is extended as follows.
Let $\mathrm{reg}_*\colon\fH^1_*\to\fH^0_*$ be the composition of the isomorphism $\fH^1_*\xrightarrow{\simeq}\fH^0_*[T]$ defining $Z^*$ and the specializing map $\fH^0_*[T]\to\fH^0_*$; $T\mapsto 0$.
The algebra homomorphism $\mathrm{reg}_{\sh}\colon\fH^1_{\sh}\to\fH^0_{\sh}$ is defined similarly.
\begin{theorem}[Extended double shuffle relation; Ihara--Kaneko--Zagier~\cite{IharaKanekoZagier2006}]\label{thm:EDSR}
For any $w_1\in\fH^1$ and $w_0\in\fH^0$, we have
\begin{equation}\label{eq:EDSR*}
Z(\mathrm{reg}_*(w_1*w_0-w_1\sh w_0))=0
\end{equation}
and
\begin{equation}\label{eq:EDSRsh}
Z(\mathrm{reg}_{\sh}(w_1*w_0-w_1\sh w_0))=0.
\end{equation}
\end{theorem}
Under the assumption of the DSR, it has been shown that \eqref{eq:regularization}, \eqref{eq:EDSR*}, and \eqref{eq:EDSRsh} are equivalent as families of abstract relations.
See \cite[Theorem~2]{IharaKanekoZagier2006}.
It is conjectured that the EDSR exhausts all algebraic relations among MZVs and thus the EDSR is considered a very important family of relations.

The proof of the EDSR by Ihara, Kaneko, and Zagier was based on analysis, but Kaneko and Yamamoto \cite{KanekoYamamoto2018} have provided an almost purely algebraic alternative proof.
However, the foundation of their proof, their ``integral-series identity,'' employs an integral (namely, Yamamoto's integral expression).

In this short note, based on a recent result (= Theorem~\ref{thm:MSW}), we provide a new proof of the EDSR employing only elementary manipulations of finite sums and observing simple behaviors of the discrete limit as $N\to\infty$. Specifically, we utilize \eqref{eq:asympH} but not \eqref{eq:asympLi}, and directly prove \eqref{eq:EDSR*}.
\subsection*{Acknowledgements}
The author would like to thank Professor Masataka Ono and Professor Shuji Yamamoto for their valuable discussions.
The author would also like to thank the anonymous referee for his/her valuable suggestions that helped improve the readability of the manuscript.
\section{Nuts and bolts}
Let $N$ be a positive integer.
For a positive integer $n$, we use the notation $[n]\coloneqq\{1,2,\dots,n\}$.
\subsection{$\zeta_{<N}^{\flat}$-value}
For an index $\bk=(k_1,\dots,k_r)$ with $\wt(\bk)=k$, we define $\zeta_{<N}^{\flat}(\bk)$ as
\[
\zeta_{<N}^{\flat}(\bk)\coloneqq\sum_{(n_1,\dots,n_k)\in S_N(\bk)}\prod_{i=1}^k\hat{\omega}_{u_i}^{(N)}(n_i),
\]
where
\[
S_N(\bk)\coloneqq\left\{(n_1,\dots,n_k)\in\ZZ^k \ \middle| \ \begin{array}{cc} n_{i-1}<n_i  & \text{ if } i\in J(\bk)\cup\{k+1\}, \\ n_{i-1}\leq n_i & \text{ if } i\in [k]\setminus J(\bk), \\ \text{where }n_0=0, & n_{k+1}=N \end{array}\right\},
\]
$\hat{\omega}^{(N)}_0(n)\coloneqq\frac{1}{n}$, $\hat{\omega}^{(N)}_1(n)\coloneqq\frac{1}{N-n}$, and if $i\in J(\bk)$, then $u_i=1$; otherwise, $u_i=0$.
($(u_i)_{i\in[k]}$ depends on $\bk$.)
The following ``finite sum = finite sum'' type identity is the key tool in this note.
\begin{theorem}[{Maesaka--Seki--Watanabe~\cite[Theorem~1.3]{MaesakaSekiWatanabe2024}}]\label{thm:MSW}
For any index $\bk$, we have
\[
\zeta_{<N}^{}(\bk)=\zeta_{<N}^{\flat}(\bk).
\]
\end{theorem}
This theorem was proved in \cite{MaesakaSekiWatanabe2024} using only series manipulation.
As explained in \cite{MaesakaSekiWatanabe2024}, this theorem can be considered as the discretization of the iterated integral expression of the MZV.
In their paper, this was applied to duality relations, but here it is applied to the EDSR.
\subsection{$R_{<N}$-value and $\zeta_{<N}^{\natural}$-value}
Let $k$ be a positive integer and $a_1, \dots, a_k$, $b_1,\dots, b_k$ non-negative integers satisfying $a_1\geq 1$ and $a_i+b_i\geq 1$ for all $i\in[k]$.
We set
\[
R_{<N}^{}(a_1,\dots, a_k; b_1,\dots, b_k)\coloneqq\sum_{0<n_1<n_2<\cdots<n_k<N}\prod_{i=1}^k\frac{1}{(N-n_i)^{a_i}n_i^{b_i}}.
\]
\begin{lemma}\label{lem:R}
\begin{enumerate}[leftmargin=*, label=(\roman*)]
\item\label{it:log}
In the general setting described above, we have
\[
R_{<N}^{}(a_1,\dots, a_k; b_1,\dots, b_k)=O(\log^kN),\quad \text{as } N\to\infty.
\]
\item\label{it:limit1}
If there exists at least one $i\in[k]$ satisfying $b_i\geq 1$ and $a_i+b_i\geq 2$, then
\[
R_{<N}^{}(a_1,\dots, a_k; b_1,\dots, b_k)=O(N^{-1}\log^kN),\quad \text{as } N\to\infty.
\]
\item\label{it:limit2}
If there exist $i$, $j\in[k]$ satisfying $i<j$, $a_i\geq 2$, and $b_j\geq 1$, then
\[
R_{<N}^{}(a_1,\dots, a_k; b_1,\dots, b_k)=O(N^{-1}\log^kN),\quad \text{as } N\to\infty.
\]
\end{enumerate}
\end{lemma}
\begin{proof}
Although there is some overlap with \cite[Lemma~2.1]{MaesakaSekiWatanabe2024} in both the statement and the proof, the proof is provided here for the convenience of the reader.
Let $c_1, d_1, \dots, c_k, d_k$ be $0$ or $1$, where for each $h\in[k]$, we can conveniently choose $c_h$ and $d_h$ such that $c_h+d_h=1$, making it possible to adapt to each of the subsequent arguments.

The estimate \ref{it:log} follows from
\begin{align*}
R_{<N}^{}(a_1,\dots, a_k; b_1,\dots, b_k)&\leq\sum_{0<n_1<\cdots<n_k<N}\prod_{i=1}^k\frac{1}{(N-n_i)^{c_i}n_i^{d_i}}\\
&\leq \sum_{0<n_1,\dots,n_k<N}\frac{1}{n_1\cdots n_k}=O(\log^kN).
\end{align*}
For \ref{it:limit1}, we assume that $b_i\geq 1$ and $a_i+b_i\geq 2$ for some $i$.
Then by an estimate
\begin{align*}
\frac{1}{(N-n_1)^{a_1}n_1^{b_1}}\cdot\frac{1}{(N-n_i)^{a_i}n_i^{b_i}}&\leq\frac{1}{N-n_1}\cdot\frac{1}{n_i}\cdot\frac{1}{(N-n_i)^{a_i}n_i^{b_i-1}}\\
&\leq\frac{1}{N-n_1}\cdot\frac{1}{n_i}\cdot\frac{1}{(N-n_i)^{c_i}n_i^{d_i}}\\
&\leq\frac{1}{(N-n_1)n_1}\cdot\frac{1}{(N-n_i)^{c_i}n_i^{d_i}}
\end{align*}
(for the case $i=1$, read as $\frac{1}{(N-n_1)^{a_1}n_1^{b_1}}\leq\frac{1}{(N-n_1)n_1}$) and the partial fraction decomposition
\[
\frac{1}{(N-n_1)n_1}=\frac{1}{N}\left(\frac{1}{N-n_1}+\frac{1}{n_1}\right),
\]
we have
\begin{align*}
&R_{<N}^{}(a_1,\dots, a_k; b_1,\dots, b_k)\\
&\leq\frac{1}{N}\left(\sum_{0<n_1<\cdots<n_k<N}\frac{1}{N-n_1}\prod_{h=2}^{k}\frac{1}{(N-n_h)^{c_h}n_h^{d_h}}+\sum_{0<n_1<\cdots<n_k<N}\frac{1}{n_1}\prod_{h=2}^{k}\frac{1}{(N-n_h)^{c_h}n_h^{d_h}}\right)\\
&=O(N^{-1}\log^kN).
\end{align*}
For \ref{it:limit2}, we assume that $a_i\geq 2$ and $b_j\geq 1$ for some $i<j$.
Then by applying the transformation $N-n_h\mapsto m_h$ and an estimate
\begin{align*}
\frac{1}{(N-m_j)^{b_j}m_j^{a_j}}\cdot\frac{1}{(N-m_i)^{b_i}m_i^{a_i}}&\leq\frac{1}{N-m_j}\cdot\frac{1}{m_i}\cdot\frac{1}{(N-m_i)^{c_i}m_i^{d_i}}\\
&\leq\frac{1}{(N-m_j)m_j}\cdot\frac{1}{(N-m_i)^{c_i}m_i^{d_i}},
\end{align*}
we have
\begin{align*}
&R_{<N}^{}(a_1,\dots, a_k; b_1,\dots, b_k)\\
&\leq\frac{1}{N}\left(\sum_{0<m_k<m_{k-1}<\cdots<m_1<N}\frac{1}{N-m_j}\prod_{h\in[k]\setminus\{j\}}\frac{1}{(N-m_h)^{c_h}m_h^{d_h}}\right.\\
&\qquad\qquad\left.+\sum_{0<m_k<m_{k-1}<\cdots<m_1<N}\frac{1}{m_j}\prod_{h\in[k]\setminus\{j\}}\frac{1}{(N-m_h)^{c_h}m_h^{d_h}}\right)\\
&=O(N^{-1}\log^kN).
\end{align*}
This concludes the proof.
\end{proof}
As examples of $R_{<N}^{}$-values not satisfying conditions \ref{it:limit1} nor \ref{it:limit2}, take $R_{<N}^{}(2,1;0,0)$ and $R_{<N}^{}(1,2;0,0)$.
The former has a finite non-zero limit $\lim_{N\to\infty}R_{<N}^{}(2,1;0,0)=\zeta(1,2)>0$ and the latter diverges $\lim_{N\to\infty}R_{<N}^{}(1,2;0,0)=\infty$.

For an index $\bk=(k_1,\dots,k_r)$ with $\wt(\bk)=k$, we define $\zeta_{<N}^{\natural}(\bk)$ as
\[
\zeta_{<N}^{\natural}(\bk)\coloneqq\sum_{(n_1,\dots,n_k)\in T_N(k)}\prod_{i=1}^k\hat{\omega}_{u_i}^{(N)}(n_i),
\]
where $T_N(k)\coloneqq\{(n_1,\dots,n_k)\in\ZZ^k \mid 0<n_1<\cdots<n_k<N\}$ and $(u_i)_{i\in[k]}$ is determined by $\bk$ in the same way as $\zeta_{<N}^{\flat}(\bk)$ was defined, i.e., if $i\in J(\bk)$, then $u_i=1$; otherwise, $u_i=0$.
\begin{proposition}\label{prop:natural}
For any index $\bk$, we have
\[
\zeta_{<N}^{\flat}(\bk)=\zeta_{<N}^{\natural}(\bk)+O(N^{-1}\log^{\bullet}N),\quad \text{as } N\to\infty.
\]
\end{proposition}
\begin{proof}
Since the right-hand side of
\[
\zeta_{<N}^{\flat}(\bk)-\zeta_{<N}^{\natural}(\bk)=\sum_{(n_1,\dots,n_k)\in S_N(\bk)\setminus T_N(k)}\prod_{i=1}^k\hat{\omega}_{u_i}^{(N)}(n_i)
\]
can be decomposed as a finite sum of $R_{<N}^{}$-values satisfying the condition \ref{it:limit1} in Lemma~\ref{lem:R}, we have the desired conclusion.
\end{proof}
Here, we define three $\QQ$-linear maps $Z_N\colon\fH^1\to\QQ$, $Z_N^{\flat}\colon\fH^1\to\QQ$, and $Z_N^{\natural}\colon\fH^1\to\QQ$ by $Z_N(1)=Z_N^{\flat}(1)=Z_N^{\natural}(1)=1$ and
\[
Z_N(e_{\bk})=\zeta_{<N}^{}(\bk),\quad Z_N^{\flat}(e_{\bk})=\zeta_{<N}^{\flat}(\bk),\quad \text{and} \quad Z_N^{\natural}(e_{\bk})=\zeta_{<N}^{\natural}(\bk).
\]
As is well known, $Z_N$ satisfies the harmonic product formula, that is, $Z_N\colon\fH^1_*\to\QQ$ is an algebra homomorphism.
On the other hand, $Z_N^{\natural}$ satisfies the shuffle product formula approximately.
\begin{proposition}[Asymptotic shuffle product formula]\label{prop:sh}
For any $w_1\in\fH^1$ and $w_0\in\fH^0$, we have
\[
Z_N^{\natural}(w_1)Z_N^{\natural}(w_0)=Z_N^{\natural}(w_1\sh w_0)+O(N^{-1}\log^{\bullet}N),\quad \text{as } N\to\infty.
\]
\end{proposition}
\begin{proof}
Let $\bk$ be an index with $\wt(\bk)=k$, not necessarily admissible, and $\bl$ an admissible index with $\wt(\bl)=l$.
The sequence $(u_i)_{i\in[k]}$ is determined as before by $\bk$, and $u_{k+j}$ is set to $1$ if $j\in J(\bl)$, otherwise $u_{k+j}$ is $0$.
Then
\begin{align*}
\zeta_{<N}^{\natural}(\bk)\zeta_{<N}^{\natural}(\bl)&=\left(\sum_{(n_1,\dots,n_k)\in T_N(k)}\prod_{i=1}^k\hat{\omega}_{u_i}^{(N)}(n_i)\right)\left(\sum_{(m_1,\dots,m_l)\in T_N(l)}\prod_{j=1}^l\hat{\omega}_{u_{k+j}}^{(N)}(m_j)\right)\\
&=\sum_{\sigma\in\Sigma_{k,l}}\sum_{(r_1,\dots,r_{k+l})\in T_N(k+l)}\prod_{h=1}^{k+l}\hat{\omega}^{(N)}_{u_h}(r_{\sigma(h)})\\
&\qquad +\sum_{\substack{(n_1,\dots,n_k)\in T_N(k), \\ (m_1,\dots, m_l)\in T_N(l), \\ n_i=m_j \text{ for some pairs } (i,j)}}\prod_{i=1}^k\hat{\omega}_{u_i}^{(N)}(n_i)\prod_{j=1}^l\hat{\omega}_{u_{k+j}}^{(N)}(m_j),
\end{align*}
where $\Sigma_{k,l}\coloneqq\{\sigma \in \mathfrak{S}_{k+l} \mid \sigma(1)<\cdots<\sigma(k), \sigma(k+1)<\cdots<\sigma(k+l)\}$ and $\mathfrak{S}_{k+l}$ denotes the group of permutations on the set $[k+l]$.
Since
\[
\sum_{\sigma\in\Sigma_{k,l}}\sum_{(r_1,\dots,r_{k+l})\in T_N(k+l)}\prod_{h=1}^{k+l}\hat{\omega}^{(N)}_{u_h}(r_{\sigma(h)}) = Z_N^{\natural}(e_{\bk}\sh e_{\bl})
\]
and
\[
\sum_{\substack{(n_1,\dots,n_k)\in T_N(k), \\ (m_1,\dots, m_l)\in T_N(l), \\ n_i=m_j \text{ for some pairs } (i,j)}}\prod_{i=1}^k\hat{\omega}_{u_i}^{(N)}(n_i)\prod_{j=1}^l\hat{\omega}_{u_{k+j}}^{(N)}(m_j)
\]
can be decomposed as a finite sum of $R_{<N}^{}$-values satisfying conditions \ref{it:limit1} or \ref{it:limit2} in Lemma~\ref{lem:R}, we have the desired conclusion.
(Note that if $n_i=m_j$ holds for some $(i,j)$, then necessarily $n_i\leq m_l$, and since $\bl$ is admissible, it follows that $u_{k+l}=0$.)
\end{proof}
\section{The proof}
\begin{theorem}[Asymptotic double shuffle relation]\label{thm:main}
For any $w_1\in\fH^1$ and $w_0\in\fH^0$, we have
\[
Z_N(w_1*w_0-w_1\sh w_0)=O(N^{-1}\log^{\bullet}N), \quad \text{as } N\to\infty.
\]
\end{theorem}
\begin{proof}
By using tools prepared in the previous section, we compute
\begin{align*}
Z_N(w_1*w_0)&=Z_N(w_1)Z_N(w_0)\\
&=Z_N^{\flat}(w_1)Z_N^{\flat}(w_0) \tag{Thm~\ref{thm:MSW}}\\
&=\bigl(Z_N^{\natural}(w_1)+O(N^{-1}\log^{\bullet}N)\bigr)\bigl(Z_N^{\natural}(w_0)+O(N^{-1}\log^{\bullet}N)\bigr) \tag{Prop~\ref{prop:natural}}\\
&=Z_N^{\natural}(w_1\sh w_0)+O(N^{-1}\log^{\bullet}N) \tag{Prop~\ref{prop:sh} + Lem~\ref{lem:R} \ref{it:log}}\\
&=Z_N^{\flat}(w_1\sh w_0)+O(N^{-1}\log^{\bullet}N) \tag{Prop~\ref{prop:natural}}\\
&=Z_N(w_1\sh w_0)+O(N^{-1}\log^{\bullet}N). \tag{Thm~\ref{thm:MSW}}
\end{align*}
This completes the proof.
\end{proof}
\begin{proof}[Proof of the EDSR~\eqref{eq:EDSR*} from Theorem~$\ref{thm:main}$]
By the asymptotic formula~\eqref{eq:asympH}, we have for some $n$,
\[
Z_N(w_1*w_0-w_1\sh w_0)=\sum_{i=0}^nZ(W_i)\cdot(\log N+\gamma)^i+O(N^{-1}\log^{\bullet}N),\quad \text{as } N\to\infty,
\]
where $W_i\in\fH^0$ for all $i\in[n]$ and $W_0=\mathrm{reg}_*(w_1*w_0-w_1\sh w_0)$.
Note that $n$ and each $W_i$ are independent of $N$.
By Theorem~\ref{thm:main}, we have
\[
\sum_{i=0}^nZ(W_i)\cdot(\log N+\gamma)^i=o(1),\quad \text{as } N\to\infty.
\]
Hence all $Z(W_i)$ should be $0$ and in particular $Z(W_0)=Z(\mathrm{reg}_{\ast}(w_1 \ast w_0 - w_1 \sh w_0)) = 0$.
\renewcommand{\qedsymbol}{Q.E.D.}
\end{proof}

\end{document}